\documentclass[11pt]{amsart}
\usepackage{geometry}                
\usepackage{graphicx}
\usepackage{amssymb}
\usepackage{epstopdf, color}
\usepackage{comment, verbatim}
\usepackage{floatflt}
\usepackage{algorithm}
\usepackage{algpseudocode}
\usepackage{amsaddr}

\theoremstyle{plain}
\newtheorem{theorem}{Theorem}[section]
\newtheorem{corollary}[theorem]{Corollary}

\newtheorem{lemma}[theorem]{Lemma}
\newtheorem{proposition}[theorem]{Proposition}

\newtheorem{definition}[theorem]{Definition}

\theoremstyle{remark}
\newtheorem{remark} [theorem]{Remark}

\newtheorem{example}[theorem]{Example}

\usepackage{multicol, amsmath, amssymb, amscd, graphicx, verbatim, color}
\usepackage{subcaption}

\newcommand{\col}{\kern -3pt :}
\newcommand{\db}{/\kern -4pt/}

\newcommand{\opp}{ \mathrm{opp}}
\newcommand{\rev}{ \mathrm{rev}}
\newcommand{\tw}{ \mathrm{tw}}

\newcommand{\C}{\mathbb C}

\newcommand{\Z}{\mathbb Z}
\newcommand{\N}{\mathbb N}

\newcommand{\skein}{\mathcal S^A}

\newcommand{\SL}{\ensuremath{\mathrm{SL} }}

\newcommand{\del}{\ensuremath{\partial}}

\newcommand{\pqvect}[2]{\ensuremath{\begin{pmatrix} {#1} \\ {#2} \end{pmatrix} } }
\newcommand{\Tvect}[2]{\ensuremath{\begin{pmatrix} {#1} \\ {#2} \end{pmatrix}_T } }
\newcommand{\smallpqvect}[2]{\ensuremath{ \left( \begin{smallmatrix} {#1} \\ {#2} \end{smallmatrix}  \right) }}
\newcommand{\smallTvect}[2]{\ensuremath{ \left( \begin{smallmatrix} {#1} \\ {#2} \end{smallmatrix}  \right)_T }}

\newcommand{\Tprimevect}[2]{\ensuremath{\begin{pmatrix} {#1} \\ {#2} \end{pmatrix}_{T'} } }
\newcommand{\smallTprimevect}[2]{\ensuremath{ \left( \begin{smallmatrix} {#1} \\ {#2} \end{smallmatrix}  \right)_{T'} }}

\newcommand{\bigD}[4] {\ensuremath D \begin{pmatrix} {#1}  &  {#3}  \\ {#2} & {#4}  \end{pmatrix}}
\newcommand{\smallD}[4] { D \left( \begin{smallmatrix} {#1}  &  {#3}  \\ {#2} & {#4}  \end{smallmatrix} \right)}

\newcommand{\qint}[1]{\ensuremath{ \left[ {#1}  \right]_{A^4} }}

\title{Fast algorithm for multiplication on the  skein algebra of one-hole torus}

\author{Siki Wang}
\address{California Institute of Technology, Pasadena, CA}
\email{siki.wang@caltech.edu}

\author{Helen Wong}
\address{Claremont McKenna College, Claremont, CA}
\email{hwong@cmc.edu}

\begin{document}

\begin{abstract}The Kauffman bracket skein algebra of a surface is a generalization of the Jones polynomial invariant for links that also plays a principal role in the Witten-Reshetikhin-Turaev topological quantum field theory.   However, the multiplicative structure of the skein algebra is not well understood, with a priori exponential complexity.  We consider the case of one-hole torus, and provide a polynomial time algorithm for computing multiplication of any two skein elements.  Some examples of closed form formulas for multiplication of curves with low crossing number are also given. 
\end{abstract}

\maketitle

{\sl 2020 Mathematics Subject Classification: 57-08, 57K31, 57K16. }

\section{Introduction}

The Kauffman bracket skein algebra of a surface  $\Sigma$ lies in the intersection of quantum topology and hyperbolic geometry.  Originally defined as a generalization of the Jones polynomial for framed links in thickened surfaces \cite{TuraevSkeinAlg, PrzSkeinAlg}, the skein algebra $\skein(\Sigma)$ plays a principal role in the skein theoretic verison of the Witten-Reshetikhin-Turaev topological quantum field theory \cite{BHMV}.   Later, it was discovered to be closely related to the $\mathrm{SL}_2 (\C)$-character variety of the fundamental group $\pi_1(\Sigma)$, which contains as real subvariety the Teichmuller space of $\Sigma$ from hyperbolic geometry \cite{ TuraevPoisson, BullockRings, BullockFrohmanJKB, PrzSikora}.  
This connection was further developed by studying the representation theory of the skein algebra \cite{BonWonQTrace, BonWonSkeinReps1, BonWonSkeinReps2, BonWonSkeinReps3, FKBUnicity, GJSUnicity, karuo2022azumaya}, which relied on development of algebraic tools to understand its algebraic structure, e.g. its center. 

In order to further develop the relationship of the skein algebra with both quantum topology and hyperbolic geometry, it is essential to understand its multiplicative structure.  Unfortunately,  the multiplicative structure of the skein algebra is completely solved for few surfaces beyond the commutative cases.  One important exception is the skein algebra of a closed torus,  for which Bullock and Przytycki computed its presentation \cite{BullockPrz} and Frohman and Gelca provided an elegant Product-to-Sum formula for the basis obtained by applying the Chebyshev polynomial of the first kind to torus knots \cite{FrohmanGelca}.  

In this paper, we investigate the one-hole torus $\Sigma_{1,1}$ and the multiplicative structure of its skein algebra $\skein(\Sigma_{1,1})$.  By computing its presentation, Bullock and Przytycki showed that  $\skein(\Sigma_{1,1})$ is  isomorphic to the quantum group $U_q(so_3)$ \cite{BullockPrz}.
  However, an explicit formula describing multiplication has been combinatorially difficult, due to the existence of the boundary.  
 Bousseau \cite{Bousseau} and Queffelec \cite{Queffelecgl2} showed strong positivity for the structural constants for multiplication in the Chebyshev basis (also known as the bracelets basis), but an analogy to the closed torus's Product-to-Sum formula remains elusive.  
 
 Here, we provide a fast algorithm to compute multiplication in the one-hole torus.   The skein algebra is spanned as a module by \emph{multicurves}, which are disjoint unions of pairwise non-intersecting, simple closed curves.   On the one-hole torus, every multicurve is a union of a $(p,q)$-torus link with some copies of the peripheral loop $\del$.   The algorithm computes the product of multicurves as a linear combination of other multicurves.  This is done using a five-term recursive relation for the discrepancy between multiplication on the one-hole torus and multiplication on the closed torus.   

\begin{theorem} \label{algthm}
There is an algorithm to compute the product of two multicurves $m, m'$ in $\skein(\Sigma_{1,1})$ in $\mathcal{O} (i(m,m')^6)$ steps, where $i(m, m')$ is the intersection number of $m$ and $m'$.  
\end{theorem}

In contrast, computing the product by resolving all crossings is $\mathcal O(2^{i(m,m')})$.  Previous algorithms such as \cite{FrohmanBloomquist} to compute multiplication on the one-hole torus relied on the Euclidean Algorithm (or equivalently, traveling on the Farey graph) and are slower.  See also \cite{bakshi4-holedsphere} for the 4-holed sphere, whose combinatorics are closely related to that of the one-hole torus.  In particular, they use an adjusted Euclidean Algorithm and  conjecture that it  is $\mathcal O(\det^{\log\det})$, where $\det$ corresponds to half the intersection number of two multicurves on the 4-holed sphere.  


In addition to Theorem \ref{algthm}, we provide closed form formulas for some cases and prove them using a five-term recursive relation.   We were unable to conjecture a simple closed form formula for multiplication for $\skein(\Sigma_{1,1})$ in all cases; however, our methods  show that the five-term recursive relation is sufficient to determine all multiplication tables and hence should also determine a closed form formula, should one be identified.  Results of our computations also confirm the positivity proven using other methods in \cite{Bousseau, Queffelecgl2}.

\section{Background and Notation}
Let $A$ be an arbitrary variable.  Let $\Sigma_{g,n}$ be an orientable surface with genus $g$ with $n$ punctures.  Then the \emph{Kauffman bracket skein algebra} $\skein(\Sigma_{g,n})$ is generated as a $\Z[A, A^{-1}]$-module by the isotopy classes of framed knots in $\Sigma \times I$, subject to the following skein relations
\begin{center}
\begin{tabular}{lll}
(1)& (Skein relation) & 
$\begin{minipage}{.4in}\includegraphics[width=\textwidth]{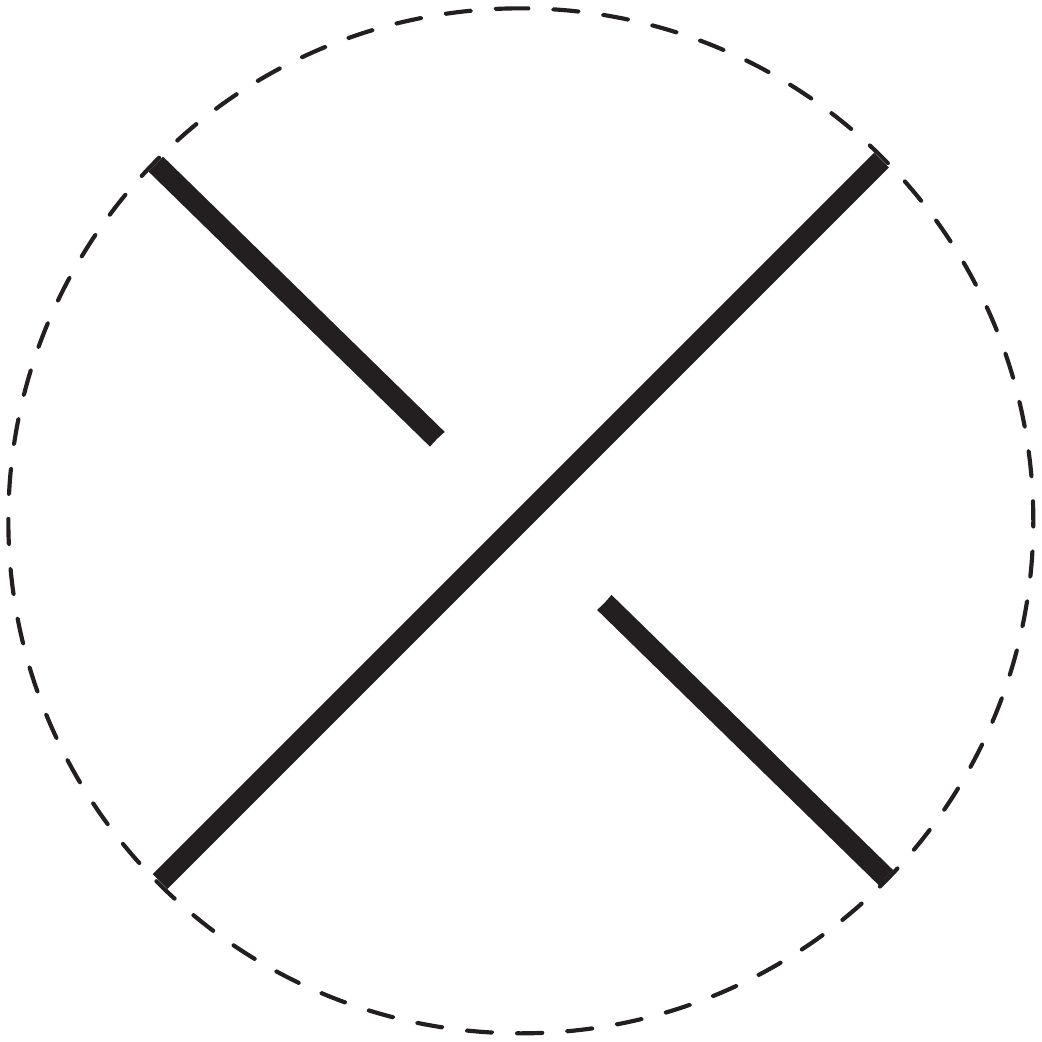}\end{minipage} 
-  
A \begin{minipage}{.4in}\includegraphics[width=\textwidth]{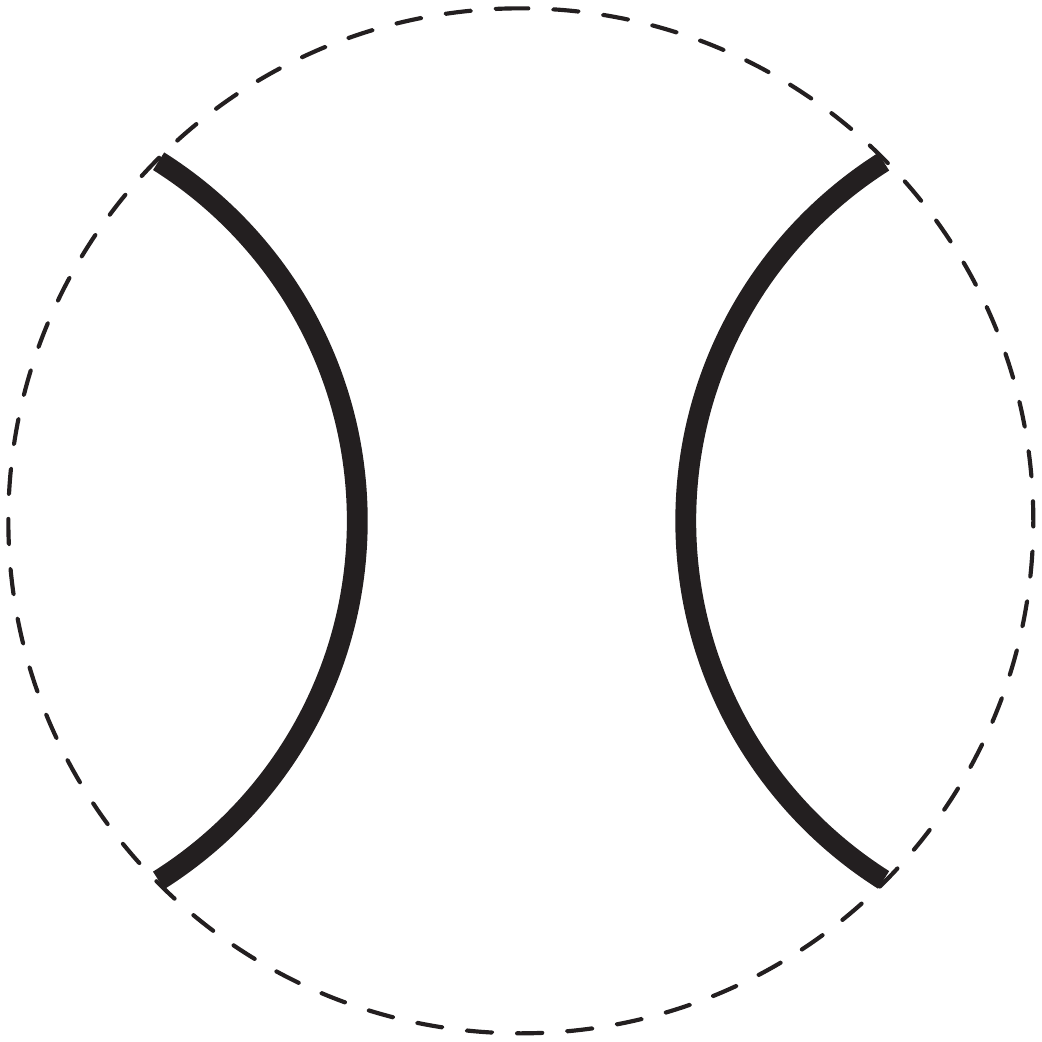}\end{minipage} 
-
A^{-1} \begin{minipage}{.4in}\includegraphics[width=\textwidth]{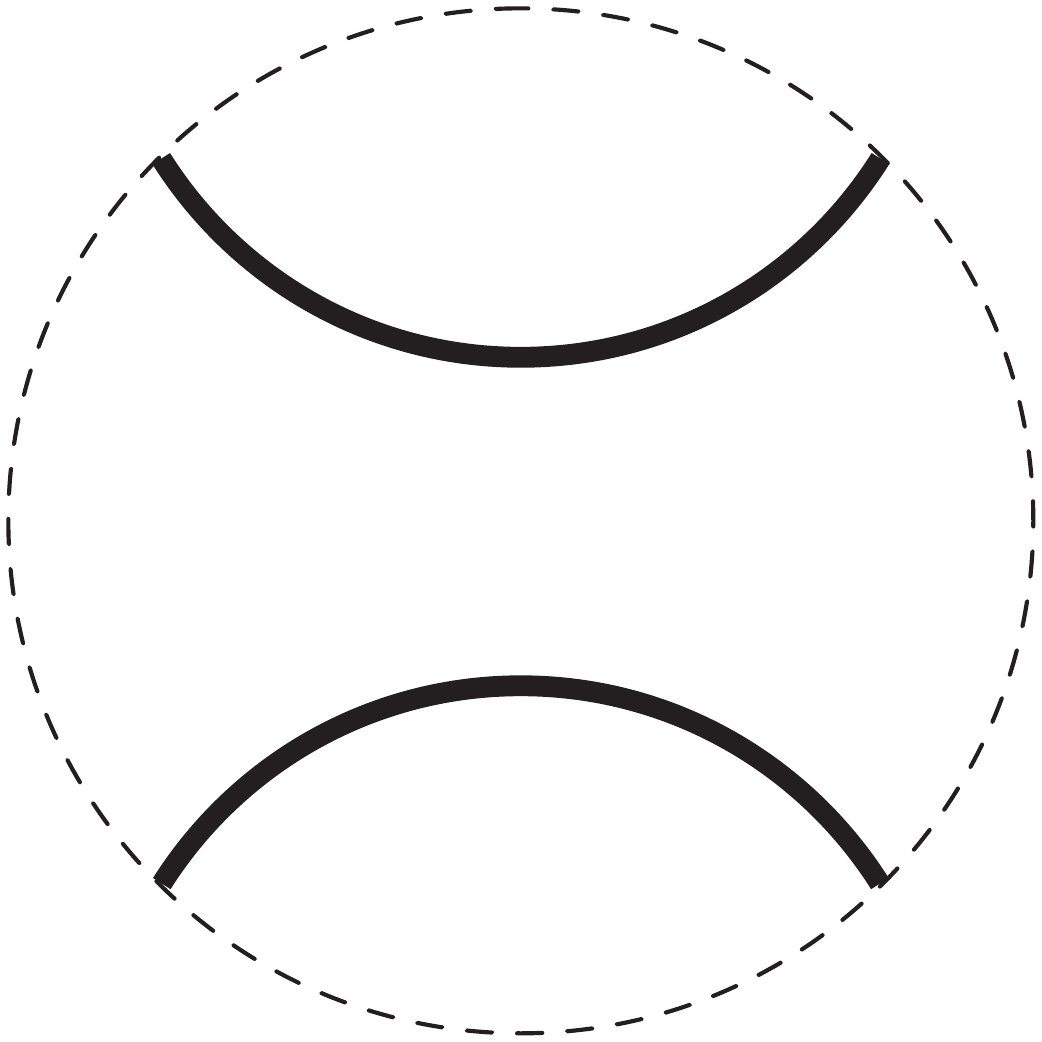}\end{minipage} $
\\[15pt]
(2)& (Framing relation) & 
$\begin{minipage}{.4in}\includegraphics[width=\textwidth]{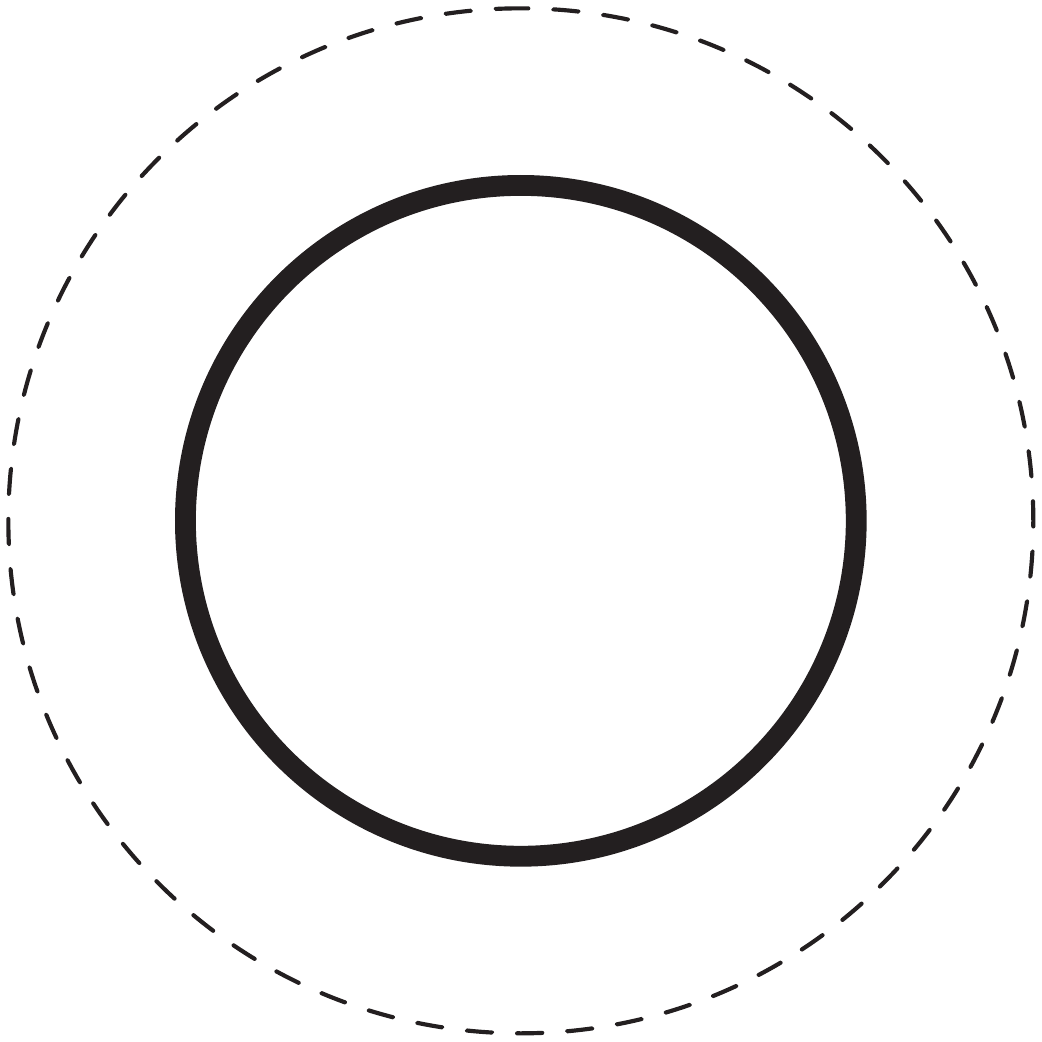} \end{minipage} 
+ A^2 + A^{-2}$\\
\end{tabular}
\end{center}
and where multiplication is by superposition.  In particular, $\alpha \cdot \beta$ has $\alpha$ on top of $\beta$.  

For any link $K$ and $i \in \N$, let $K^{(i)}$ be the $i$-fold cable of $K$, obtained by taking $i$ parallel copies $K$ in the direction of the framing.  One can extend this linearly.  If $p(x) \in Z[A, A^{-1}] [x] $ is a polynomial in $x$ such that $p= \sum_{i} c_i x^i$, then let the \emph{threading of $K$ by $p(x)$} be the skein
\[ K^p = \sum_i c_i K^{(i)}\]
Note that when $K$ has no crossings, then $K^p = p(K)$.

In this paper, we will primarily be concerned with threading by Chebyshev polynomials of the first kind $T_n$, which are defined recursively using: 
$T_0(x) = 2$, $T_1(x) = x$, and $T_{k}(x) = x T_{k-1}(x) - T_{k-2}(x)$ for $k \geq 2$.    
Note that these polynomials are closely related to the recursion of cosines and satisfy the identity $ T_k(2\cos \theta) = 2\cos(k \theta)$. 

We begin by a pair of illustrative examples that explains some of the combinatorial difficulty of the one-hole torus $\Sigma_{1,1}$ in contrast to the closed torus $\Sigma_{1,0}$. 

\begin{example} \label{ex:closed}
Let $\Sigma = \Sigma_{1,0}$ the closed torus, which we represent as a square with its opposite sides identified. 
 For every pair $(p, q) \in \Z^2 \setminus \{(0,0) \}$, let $\smallpqvect{p}{q}$ denote the torus link with $\gcd(p, q)$ components, all with slope $q/p$. 

In the closed torus, 
 \begin{eqnarray*}
\smallpqvect{2}{1} \cdot \smallpqvect{0}{1} 
&\!\!= \!\!& \begin{minipage}{.5in}\includegraphics[width=\textwidth]{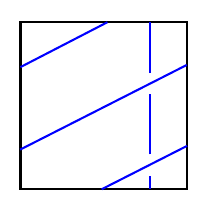}\end{minipage} 
\\
&\!\!= \!\!&A^2 \begin{minipage}{.5in}\includegraphics[width=\textwidth]{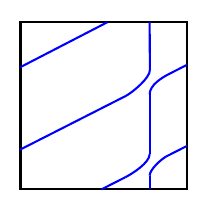}\end{minipage} 
+  \begin{minipage}{.5in}\includegraphics[width=\textwidth]{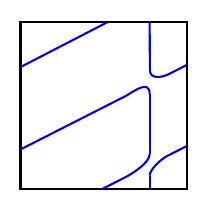}\end{minipage} 
+  \begin{minipage}{.5in}\includegraphics[width=\textwidth]{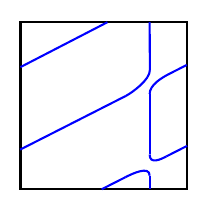}\end{minipage} 
+ A^{-2} \begin{minipage}{.5in}\includegraphics[width=\textwidth]{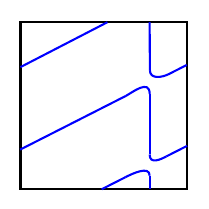}\end{minipage} 
\\
&\!\!= \!\!& A^2 \smallpqvect{2}{2} + (-A^2- A^{-2}) + (-A^2- A^{-2}) + A^{-2} \smallpqvect{2}{0}
\\
&\!\!= \!\!& A^2 (\smallpqvect{2}{2} -2)  + A^{-2}( \smallpqvect{2}{0} -2) \\
&\!\!= \!\!& A^2 \, T_2(\smallpqvect{1}{1})  + A^{-2} \, T_2( \smallpqvect{1}{0}) 
\end{eqnarray*}
where we've simplified the computation by rewriting using Chebyshev polynomials.  Since there are 2 crossings, $2^2$ resolutions result, of which two contain curves bounding disks in the closed torus.  
\end{example}

\begin{example} \label{ex:holed} 
Now let $\Sigma = \Sigma_{1,1}$ the one-hole torus, which we represent as a square with a neighborhood of its corners removed and opposite sides identified. 
With abuse of notation,  let $\smallpqvect{p}{q}$ again denote the $(p,q)$-torus link in the one-hole torus using the convention established above for the closed torus.  In addition, for $k \in \N$, let $\del$ be the peripheral loop, i.e. the isotopy class of loop parallel to the boundary.  

In the one-hole torus, 
 \begin{eqnarray*}
\smallpqvect{2}{1} \cdot \smallpqvect{0}{1} 
&\!\!= \!\!& \begin{minipage}{.5in}\includegraphics[width=\textwidth]{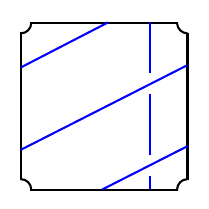}\end{minipage} 
\\
&\!\!= \!\!&A^2 \begin{minipage}{.5in}\includegraphics[width=\textwidth]{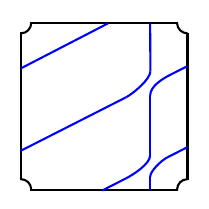}\end{minipage} 
+  \begin{minipage}{.5in}\includegraphics[width=\textwidth]{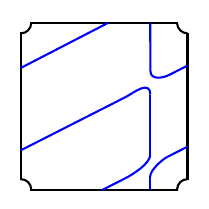}\end{minipage} 
+  \begin{minipage}{.5in}\includegraphics[width=\textwidth]{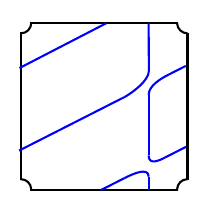}\end{minipage} 
+ A^{-2} \begin{minipage}{.5in}\includegraphics[width=\textwidth]{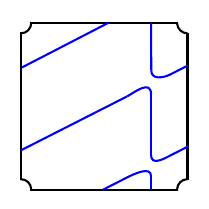}\end{minipage} 
\\
&\!\!= \!\!& A^2 \smallpqvect{2}{2} + \del   + (-A^2- A^{-2}) + A^{-2} \smallpqvect{2}{0}
\\
&\!\!= \!\!& A^2 (\smallpqvect{2}{2} -2)  + A^{-2}( \smallpqvect{2}{0} -2) + (\del + A^2 + A^{-2}) 
\\
&\!\!= \!\!& A^2 T_2(\smallpqvect{1}{1})  + A^{-2} T_2(\smallpqvect{1}{0}) + (\del + A^2 + A^{-2}) 
\end{eqnarray*}
In contrast to the closed torus case, one of the four resolutions involves a curve bounding a disk (so can be replaced by $- A^2 - A^{-1}$), and one of the resolutions involves  $\del$.  Determining when a resolution contains loops bounding a disk and when a resolutions contain copies of $\del$ presents the main combinatorial difficulty in understanding the algebraic structure of $\skein(\Sigma_{1,1})$.  As we will encounter later in Example \ref{ex:10401}, the discrepancy between the closed torus case and the one-holed torus case can be quite complicated. 
\end{example}

In this paper, we focus on the case where $\Sigma= \Sigma_{1,1}$.  Every multicurve is a union of a $(p,q)$-torus link with parallel copies of the peripheral loop $\del$.   Hence the skein algebra $\skein(\Sigma_{1,1})$ is spanned as a $\Z[A, A^{-1}]$-module by the set of multicurves $\mathcal M = \{\del^k \smallpqvect{p}{q} \}_{ k \in \N, (p, q) \in \Z^2} $.    

Note that we follow the convention $\smallpqvect{p}{q} = \smallpqvect{-p}{-q}$, and unless otherwise specified, we will choose $p\geq 0$.  Also, note that $\smallpqvect{0}{0} = \emptyset$, whereas $\del$ is the peripheral loop.  Both $\emptyset$ and $\del$ are central elements of the skein algebra.  Also remark that the crosing number   $i(m, m')$  between two multicurves $m,  m' \in \mathcal M$ is easily computed.     When the multicurves are torus links $\smallpqvect{p}{q}$ and $\smallpqvect{r}{s}$,  then $i\left( \smallpqvect{p}{q}, \smallpqvect{r}{s} \right) = | \begin{smallmatrix} p & r\\q&s \end{smallmatrix}| = ps-rq$.  We also have $ i(\emptyset, m) = 0 $ and  $i(\del^k, m) = 0 $ for all $m \in \mathcal M$.

Following the lead of Frohman and Gelca, define \[ \Tvect{p}{q} = T_{\gcd(p, q)} \pqvect{\frac{p}{\gcd(p, q)}}{\frac{q}{\gcd(p, q)}}. \]    
Note that $\smallTvect{0}{0} = 2 \emptyset $.   Later in Section \ref{sec:closedformulas} we will also use variations $T'$ and $\smallTprimevect{p}{q}$.

\begin{example}
Since $\gcd(2,2) = 2$ and $T_2(x) = x^2 - 2$, we have  $\smallTvect{2}{2} =T_2( \smallpqvect{1}{1}) = \smallpqvect{2}{2} - 2$.  
\end{example}

Let  $\eta= \del + A^2 + A^{-2}$, and let $\mathcal M_T  = \{\emptyset \} \cup \{ \eta^k \smallTvect{p}{q} \}$ where $k \in \N$ and $(p, q) \in \Z^2$.  $\mathcal M_T$ spans $\skein(\Sigma_{1,1})$ as a $Z[A, A^{-1}]$-module, and $\eta$ is a central element.   For reasons explained below, $\mathcal M_T$ will be our basis of choice.

As we saw from the calculations in Examples \ref{ex:closed} and \ref{ex:holed}, $\skein(\Sigma_{1,1})$ is closely related to $\skein(\Sigma_{1,0})$, but there are significant differences.  In particular, the one-hole torus does not satisfy Frohman-Gelca's Product-to-Sum; that is, $\smallTvect{p}{q} \smallTvect{r}{s} \neq A^{| \begin{smallmatrix} p & r\\q&s \end{smallmatrix}|} \smallTvect{p+r}{q+s} + A^{- | \begin{smallmatrix} p & r\\q&s \end{smallmatrix}| } \smallTvect{p-r}{q-s} $ in $\skein(\Sigma_{1,1})$.  This discrepancy is the subject of our paper.


\begin{proposition}\label{FrohmanGelca}
In $\skein(\Sigma_{1,1}) $, 
$  \Tvect{p}{q} \cdot \Tvect{r}{s} - A^{| \begin{smallmatrix} p & r\\q&s \end{smallmatrix}|} \pqvect{p+r}{q + s}_T 
		-  A^{- | \begin{smallmatrix} p & r\\q&s \end{smallmatrix} |} \pqvect{p-r}{q - s}_T$
is divisible by		 $\eta= \del + A^2 + A^{-2}$  .   
\end{proposition} 
\begin{proof}
In \cite{BullockPrz}, it was shown that the embedding $\Sigma_{1,1}  \hookrightarrow \Sigma_{1, 0}$ induces a surjective algebra homomorphism $ \iota: \skein(\Sigma_{1,1}) \to \skein(\Sigma_{1, 0})$ and $\ker(\iota)$ is the principal  ideal  generated by~$\eta$.  
In $\skein(\Sigma_{1, 0})$,   Frohman and Gelca's  Product-to-Sum formula 
$\smallTvect{p}{q} \cdot \smallTvect{r}{s} - A^{| \begin{smallmatrix} p & r\\q&s \end{smallmatrix}|} \smallpqvect{p+r}{q + s}_T 
		-  A^{- | \begin{smallmatrix} p & r\\q&s \end{smallmatrix} |} \smallpqvect{p-r}{q - s}_T = 0$ holds, and 
		 the proposition follows.   
\end{proof}

\begin{definition} 
For any $p, q, r, s \in \Z$, the \emph{Frohman-Gelca discrepancy}  for $\Tvect{p}{q} \cdot \Tvect{r}{s} $  is defined to be the skein
\[ \bigD{p}{q}{r}{s}=   \frac{1}{\eta} \left( \Tvect{p}{q} \cdot \Tvect{r}{s} - A^{| \begin{smallmatrix} p & r\\q&s \end{smallmatrix}|} \pqvect{p+r}{q + s}_T 
		-  A^{- | \begin{smallmatrix} p & r\\q&s \end{smallmatrix} |} \pqvect{p-r}{q - s}_T \right) \]
		where $\eta =  \del + A^2 + A^{-2} $. 
\end{definition} 

Alternatively, we have
\begin{equation} \label{def:D}
 \Tvect{p}{q} \cdot \Tvect{r}{s} 
 =  A^{| \begin{smallmatrix} p & r\\q&s \end{smallmatrix}|} \pqvect{p+r}{q + s}_T 
 +   A^{- | \begin{smallmatrix} p & r\\q&s \end{smallmatrix} |} \pqvect{p-r}{q - s}_T  
 + \eta \, \bigD{p}{q}{r}{s}.  
 \end{equation}
Computing the  Frohman-Gelca discrepancy $D$ immediately yields the structural constants for multiplication in the Chebyshev basis $\mathcal M_T$ for $\skein(\Sigma_{1,1})$.  A priori, computation of $\smallD{p}{q}{r}{s}$ involves at least $2^c$ resolutions, where $c=ps-rq$, but we will show that it can be computed in far fewer steps via a recursive process based on Theorem \ref{assoc}.

Note that the discrepancy need not be a scalar, but is in general a skein.  

\begin{example}  \label{ex:10401}
\[ \bigD{10}{4}{0}{1} = A^{-2} \pqvect{8}{3}_T + A^{-6} \qint{2} \pqvect{6}{3}_T+ A^{-4} \qint{2}\pqvect{4}{1}_T + A^2 ( \eta+ \qint{4}) \pqvect{2}{1}_T, 	\] 
where $\qint{ n} =  \dfrac{ A^{2n} - A^{-2n}} {A^2 - A^{-2}}$ for integers $n \geq 1$.  
\end{example}

We leave the proofs of the next two lemmas as exercises for the reader. 
\begin{lemma} \label{smallcrossings}
If $|ps-rq| \in \{ 0, 1 \} $, then $\bigD{p}{q}{r}{s} = 0$. 
\end{lemma}

\begin{lemma} \label{smallcrossings2}
 $\bigD{2}{0}{0}{1} = 0$ and  $\bigD{2}{1}{0}{1} = 1$
\end{lemma} 

The following is also well-known, so we will only provide a sketch proof.   

\begin{proposition} \label{q=0}
 $D \begin{pmatrix} p& 0\\0&1 \end{pmatrix}   = 0 $
 \end{proposition} 
\begin{proof} 
 By definition, 
$ \smallpqvect{p}{0}_T \cdot \smallTvect{0}{1} =  T_p \left(  \smallpqvect{1}{0}  \right) \cdot \smallpqvect{0}{1} $ is a linear combination of products of torus links $\smallpqvect{k}{0} \cdot \smallpqvect{0}{1}$ where $0 \leq k \leq p$.   
In each $\smallpqvect{k}{0} \cdot \smallpqvect{0}{1}$, there are $k$ vertical crossings, which separate  the one-hole torus into $k$ distinct regions.  After resolution of the crossings by the skein relation, the regions are all connected together, and each resolution yields a single non-separating loop, which cannot be $\del$. Since no term of $ \smallpqvect{p}{0}_T \cdot \smallTvect{0}{1}$ contains $\del$,  the discrepancy is zero. 
\end{proof}

\section{The Recursion Rule and transformations on the one-hole torus}

In this section, we gather the tools needed to compute the discrepancy.  The main innovation in this paper is the Recursion Rule.  A special case appears in unpublished work \cite{FrohmanBloomquist}, but to the knowledge of the authors, this general version of the Recursion Rule  is new. All other tools in this section are previously known.

\subsection{The Recursion Rule}

\begin{theorem}[Recursion Rule] \label{assoc}
\begin{align*}
\begin{pmatrix} u\\v \end{pmatrix}_T \cdot D\begin{pmatrix} p&r\\q&s \end{pmatrix}
& + A^{|\begin{smallmatrix}p&r\\q&s \end{smallmatrix}|} D\begin{pmatrix} u & p+r\\v & q+s \end{pmatrix} + A^{- | \begin{smallmatrix}p&r\\q&s \end{smallmatrix}|} D \begin{pmatrix} u & p-r\\v & q-s \end{pmatrix} \\
&= D \begin{pmatrix} u & p\\v & q \end{pmatrix}  \cdot \begin{pmatrix} r\\s \end{pmatrix}_T
+ A^{|\begin{smallmatrix}u&p\\v&q \end{smallmatrix}|} D \begin{pmatrix} u+p & r\\v+q & s \end{pmatrix} + A^{-|\begin{smallmatrix}u&p\\v&q \end{smallmatrix}|} D \begin{pmatrix} u-p & r\\v-q & s \end{pmatrix}
\end{align*}
\end{theorem}

\begin{proof}
Expand the  associativity rule  $\smallpqvect{u}{v}_T \cdot  \left( \smallpqvect{p}{q}_T \cdot \smallTvect{r}{s} \right) = \left( \smallpqvect{u}{v}_T \cdot   \smallpqvect{p}{q}_T \right) \cdot \smallTvect{r}{s} $.  Applying   Equation (\ref{def:D}), the result follows immediately.  
\end{proof}

\subsection{Reductions under Dehn twists} \label{mcg}

The mapping class group of the one-hole torus is isomorphic to $\SL(2, \Z)$, and generated by Dehn twists.  A mapping class $\phi$  induces an action~$\phi_*$ on $\Sigma_{1,1}$.  Specifically,  if $\phi$ is represented by the matrix $\begin{pmatrix}a & c\\ b& d \end{pmatrix} \in \SL(2, \Z)$, then $\phi_*$ acts by matrix multiplication on the torus links, with $\phi_* \pqvect{p}{q} = \begin{pmatrix}a & c\\ b& d \end{pmatrix} \cdot \pqvect{p}{q} = \pqvect{ap + cq}{bp + dq}$.  Since $\phi$ fixes the puncture, $\phi_*(\eta) = \eta$.   It immediately follows that $\phi_*$ also acts on the discrepancy.  

\begin{lemma}[Reduction by Mapping Class] \label{Dehn}
Let a mapping class $\phi$ be represented by  $\begin{pmatrix}a & c\\ b& d \end{pmatrix} \in \SL(2, \Z)$. Then for any $p, q, r, s \in \Z$, 
\[ \phi_* \left(\bigD{p}{q}{r}{s} \right)=  D\left( \phi_*  \begin{pmatrix}p & r\\ q& s \end{pmatrix} \right) = D \left( \begin{matrix} ap + cq & cr + ds\\ ar + cs &br + ds \end{matrix} \right).\]  
\end{lemma}

By the Euclidean Algorithm, we may always find $\phi$ so that $\phi_* \left( \bigD{p_0}{q_0}{r_0}{s_0} \right) = \bigD{p}{q}{0}{s}$ where $ 0\leq q < p$.  Since $\smallpqvect{0}{s} = \smallpqvect{0}{-s}$, we may assume  $s \geq 0$.  Note that if $s=0$ or $q = 0$, then the discrepancy is zero by Lemmas \ref{smallcrossings} and \ref{q=0}.   So from now on, we will take $0< q <p$ and $r=0$, and we will split our analysis into the cases where $s=1$ and where $s \geq 2$. For future reference, we  give the forms of the Recursion Rule that we will apply to these two cases.    

\begin{corollary}[Recursion Rule to increase $p$ when $r=0$ and $s = 1$] \label{1-recursion}
For $p, q \in \Z$, 
\begin{eqnarray*}
 \bigD{p+1}{q}{0}{1} 
  &=& A^{-q} \Tvect{1}{0} \cdot  \bigD{p}{q}{0}{1} - A^{-2q} \bigD{p-1}{q}{0}{1}  \\
 & &   + A^{-p-q} \bigD{1}{0}{p}{q-1}   - A^{-q} \bigD{1}{0}{p}{q} \cdot \Tvect{0}{1}       + A^{p-q} \bigD{1}{0}{p}{q+1}  
 \end{eqnarray*}
\end{corollary}

\begin{corollary}[Recursion Rule to increase $s$ when $r=0$] \label{s-recursion}
For any $p, q, s \in \Z$, 
\begin{eqnarray*}
\bigD{p}{q}{0}{s+1} &=& \bigD{p}{q}{0}{s} \cdot \Tvect{0}{1} - \bigD{p}{q}{0}{s-1} \\
& & + A^{ps} \bigD{p}{q+s}{0}{1} +A^{-ps} \bigD{p}{q-s}{0}{1}  
\end{eqnarray*} 

\end{corollary}

\subsection{Reducing discrepancy using symmetries} \label{transf}

Symmetries of $\Sigma_{1,1}$, especially rotations and  reflections,   induce simple linear algebraic transformations in the torus knots basis for the skein algebra, and these carry over to the discrepancy.  
We introduce some language to describe some of the actions that we use most frequently.

\begin{definition}
For a skein  $K =  \displaystyle \sum_{\mathcal M_T}  c_{k,i,j} \cdot \eta^k \Tvect{i}{j}$ with  $c_{k, i, j} \in \Z[A, A^{-1}]$,
define:
\\
 its \emph{twist} to be 
\[ \mathrm{tw}(K) = \displaystyle \sum_{\mathcal M_T}  c_{k,i,j} \cdot  \eta^k \Tvect{i}{j+i}, \]
 its \emph{opposition} to be 
\[ \mathrm{opp}(K) = \displaystyle \sum_{\mathcal M_T}  \overline{c_{k,i,j}} \cdot  \eta^k \Tvect{i}{i-j},  \]
and  its \emph{reverse} to be 
\[ \mathrm{rev}(K) = \displaystyle \sum_{\mathcal M_T}  c_{k,i,j} \cdot  \eta^k \Tvect{j}{i}, \]
where $\overline{c_{k, i, j}}$ denotes the conjugate, where $A$ and $A^{-1}$ are exchanged.  
\end{definition}


Remark that $\tw^{-1} \left(  \sum_{\mathcal M_T}  c_{k,i,j} \cdot \eta^k \Tvect{i}{j} \right) = \sum_{\mathcal M_T}  c_{k,i,j} \cdot \eta^k \Tvect{i}{j-i}$ whereas $\opp = \opp^{-1}$ and $\rev = \rev^{-1}$.

\begin{lemma}[Twist] \label{tw}
$\tw( \bigD{p}{q}{0}{1}) = \bigD{p}{p+q}{0}{1} $ 
\end{lemma}
\begin{proof}
This is a direct corollary of Lemma \ref{Dehn}  with $\begin{pmatrix}1 & 0\\1& 1 \end{pmatrix} \in \SL(2, \Z)$. 
\end{proof}

\begin{lemma}[Opposition] \label{opp}
$\opp( \bigD{p}{q}{0}{1}) = \bigD{p}{p-q}{0}{1} $ 
\end{lemma} 
 
\begin{proof}
Represent the one-holed torus as a square with opposite edges identified, with the hole at the corners.  Reflect vertically, across the $yz$-plane, and then apply a Dehn twist.
\end{proof}

\begin{lemma}[Reverse]  \label{reverse}
 $  \rev(\bigD{1}{0}{q}{p} )=  \bigD{p}{q}{0}{1}$ \\

\end{lemma} 

\begin{proof}
  Starting from $\smallpqvect{p}{q} \cdot \smallpqvect{0}{1}$, reflect across vertical $yz$-plane of slope, then rotate counterclockwise around the central $z$-axis of the square by $90^\circ$, and finally reflect in the central horizontal plane to obtain $\smallpqvect{1}{0} \cdot  \smallpqvect{q}{p}$.    Reflection vertically  takes a curve of slope $q/p$ to one of slope $-q/p$, and exchanges positive and negative crossings so that the coefficients $A$ and $A^{-1}$ are also exchanged.    Rotating by $90^\circ$  around the center takes a curve of slope $q/p$ to one of slope $-p/q$, and the coefficients remain unchanged.   Reflection in the plane of the square changes all over-and under-crossings.  The order of multiplication is switched, and $A$ and $A^{-1}$ are exchanged.   \end{proof}

\begin{corollary} \label{extraterms} Suppose that $q > 0 $. If  $p = aq + c$ with $a \geq 0$ and $0 \leq c < q$, then \\
$\bigD{1}{0}{p}{q} = 
 \rev \circ \tw^{a}\left( \bigD{q}{c}{0}{1} \right)  $  

\end{corollary}

\section{Some example calculations}  \label{sec:closedformulas}
In certain cases, we can obtain closed form formulas for discrepancy.  
Many of these formulas are known to experts in this area, but we present them here both as an application of the Recursion Rule as well as to establish base cases for our algorithm later.

In the following, we will use the following notation for the  quantum integer 
\[ \qint{ k} =  \dfrac{ A^{2k} - A^{-2k}} {A^2 - A^{-2}} \mbox{ for integers } k \geq 1.\]  
 Also, the closed form formulas are simpler if we use a commonly used variation of the Chebyshev polynomial.  Define $T'_k$ so that $T'_0(x) = 1$ but otherwise $T'_k(x) = T_k(x)$ for all $k \geq 1$. 
Then we have that 
\[ \Tprimevect{0}{0} = \emptyset, \quad  \mbox{  and} \Tprimevect{p}{q} = \Tvect{p}{q} \mbox{ whenever }\pqvect{p}{q} \neq \pqvect{0}{0}.\]

\begin{proposition} \label{q=1} 
For $p \geq 1$, 
\[  \bigD{p}{1}{0}{1}
= \sum_{k=0}^{\lfloor \frac{p}{2} \rfloor}  A^{-(p-2k)} \qint{k}  \Tprimevect{p-2k}{0}
\]

  
\end{proposition}


\begin{proof}
By induction.  Lemmas \ref{smallcrossings} and \ref{smallcrossings2} give the base cases $p=1,2$.  
From the Recursion Rule and Lemma \ref{smallcrossings}, we get the identity for the induction step
\begin{eqnarray*}
 \smallD{p}{1}{0}{1} 
  &=& A^{-1} \smallpqvect{1}{0} \cdot  \smallD{p-1}{1}{0}{1} - A^{-2} \smallD{p-2}{1}{0}{1}           + A^{p-2} \smallD{1}{0}{p-1}{2}. 
\end{eqnarray*}
Note that we have $\smallD{1}{0}{p-1}{2} = \begin{cases} 1, & p  \mbox{ even} \\ 0, &p \mbox{ odd}  \end{cases}$ from Lemmas \ref{smallcrossings2} and \ref{reverse}.  

\end{proof}

As a consequence, we may apply twist, opposition, and reverse to get other values of the discrepancy.

  For example, 
\begin{equation} \label{q=p+1} 
 \bigD{p}{p+1}{0}{1}
= \sum_{k=0}^{\lfloor \frac{p}{2} \rfloor}  A^{-(p-2k)} \qint{k}  \Tprimevect{p-2k}{p-2k}
\mbox{ for } p \geq 1
\end{equation}
\begin{equation} \label{q=p-1} 
 \bigD{p}{p-1}{0}{1}
= \sum_{k=0}^{\lfloor \frac{p}{2} \rfloor}  A^{+(p-2k)} \qint{k}  \Tprimevect{p-2k}{p-2k}
\mbox{ for } p \geq 1
\end{equation}
and 
\begin{equation} \label{q=-1} 
 \bigD{p}{-1}{0}{1}
= \sum_{k=0}^{\lfloor \frac{p}{2} \rfloor}  A^{+(p-2k)} \qint{k}  \Tprimevect{p-2k}{0}
\mbox{ for } p \geq 1
\end{equation}


We may also obtain other values of the discrepancy by using the Resursion Rule. 
By Corollary \ref{s-recursion}, 
\begin{eqnarray*}
 \bigD{p}{0}{0}{2} 
 &=& A^p \smallD{p}{1}{0}{1} + A^{-p} \smallD{p}{-1}{0}{1} \\
 &=&
 = A^p \sum_{k=0}^{\lfloor \frac{p}{2} \rfloor}  A^{-(p-2k)} \qint{k}  \smallTprimevect{p-2k}{0}
 + A^{-p}  \sum_{k=0}^{\lfloor \frac{p}{2} \rfloor}  A^{+(p-2k)} \qint{k}  \smallTprimevect{p-2k}{0}
\\
&=& \sum_{k=0}^{\lfloor \frac{p}{2} \rfloor}   \qint{2k}  \Tprimevect{p-2k}{0}.
 \end{eqnarray*}

We could repeat Corollary \ref{s-recursion} as below
\begin{eqnarray*}
 \bigD{p}{0}{0}{3} 
 &=&  \smallD{p}{0}{0}{2} \smallTvect{0}{1} + A^{2p} \smallD{p}{1}{0}{1} + A^{-2p} \smallD{p}{-1}{0}{1} \\
 &=&
 \sum_{k=0}^{\lfloor \frac{p}{2} \rfloor}   \qint{2k}  \smallTprimevect{p-2k}{0} \cdot \smallTvect{0}{1} 
 + 
(  A^{p +2k} + A^{-(p+2k)} )  \qint{k}  \smallTprimevect{p-2k}{0}
\\
&=& \sum_{k=1}^{\lfloor \frac{p}{2} \rfloor}  
 \qint{2k} 
 \left( A^{p-2k} \smallTvect{p-2k}{1} +  A^{-(p-2k)} \smallTvect{p-2k}{-1} +  \eta \smallD{p-2k}{0}{0}{1}  \right) 
\\
&&\qquad  + 
(  A^{p +2k} + A^{-(p+2k) } )  \qint{k}  \smallTprimevect{p-2k}{0}
\\
&=& \cdots, 
 \end{eqnarray*}
but the formulas become unwieldy quickly. 




One may also apply the Recursion Rule Corollary \ref{1-recursion} to find $\smallD{p}{q}{0}{1}$ for $q \geq 2$. However, the Recursion Rule for $\smallD{p+1}{q}{0}{1}$ involves three differing terms that depend on $p$ mod $q-1$, $q$, and $q+1$, respectively.  As an example, we see below that the formula for $\smallD{p}{2}{0}{1}$ depends on $p$ modulo $6$. 

\begin{proposition} \label{q=2} 
 $\bigD{2}{2}{0}{1}  = 0$, \;   $ \bigD{3}{2}{0}{1} = A \Tvect{1}{1}$ 
 and for $p \geq 4$, 
 \[
  \bigD{p}{2}{0}{1} 
  =  \sum_{k=1}^{ \lfloor \frac{p}{3} \rfloor} A^{-(p-4k)}  \qint{k} \Tvect{p-2k}{1}
  + \begin{cases} \sum_{k=1}^{\lfloor \frac{p-1}{6} \rfloor} A^{2k} \qint{2k} \Tvect{2k}{1} & p = 0 \pmod{2}\\
  \sum_{k=1}^{\lfloor \frac{p + 1}{6} \rfloor} A^{2k-1} \qint{2k-1} \Tvect{2k-1}{1} & p = 1 \pmod{2}
  \end{cases}\]

\end{proposition}

\begin{proof} 
The base case $p=2$ follows from Lemma \ref{smallcrossings2}, and the case $p=3$ follows from Proposition \ref{q=2}, where $\smallD{3}{1}{0}{1}=  A^{-1} \smallTvect{1}{0}$ implies $\smallD{ 3}{2}{0}{1}  =  A \smallTvect{1}{1}$. 
The Recursion Rule in this case is
\begin{eqnarray*}
  \bigD{p}{2}{0}{1} 
  &=& A^{-2} \pqvect{1}{0} \cdot  \bigD{p-1}{2}{0}{1} - A^{-4} \bigD{p-2}{2}{0}{1}   \\
  && \qquad - A^{-2} \bigD{1}{0}{p-1}{2} \cdot \pqvect{0}{1}       + A^{p-3} \bigD{1}{0}{p-1}{3}
    \end{eqnarray*}
where we have
\[ \bigD{1}{0}{k}{2} \cdot \Tvect{0}{1} = 
\begin{cases} 
0  & k = 0 \pmod{2}\\
\Tvect{0}{1} & k = 1  \pmod{2} 
\end{cases} 
\] 
and 
\[  \bigD{1}{0}{p}{3} = 
\begin{cases}
0, & p =0 \pmod{3} \\
A^{-1} \Tvect{\frac{p-1}{3}}{1} , & p =1 \pmod{3}  \\
A \Tvect{\frac{p+1}{3}}{1}, & p =2 \pmod{3}   \\
\end{cases}
\] 
from previous calculations.  The rest of the proof follows. 
\end{proof}

Thus far, we have taken the Recursion Rule with one factor being $\smallTvect{u}{v} = \smallpqvect{0}{1}$, but we can use $\smallTvect{u}{v} = \smallpqvect{2}{1}$ to see other patterns.  
\begin{proposition} \label{p=2q} 
For $q \geq 1$, 
\[
  \bigD{2q}{q}{0}{1} 
  =  
\sum_{k=1}^{\lfloor \frac{q+1}{2} \rfloor}  \qint{2k-1} \Tprimevect{2(q- 2k+1)}{q-2k+1}  
\]
\end{proposition} 

\begin{proof}
We have that $ \bigD{0}{0}{0}{1} = 0$ and $ \bigD{2}{1}{0}{1} = 1$.  All the rest can be computed using the Recursion Relation, 
\begin{eqnarray*}
 \smallD{2(q+1)}{q+1}{0}{1} 
 &=& \smallTvect{2}{1} \smallD{2q}{q}{0}{1}  -  \smallD{2(q-1)}{q-1}{0}{1}  + A^{2q}  \smallD{2}{1}{2q}{q+1} + A^{-2q}  \smallD{2}{1}{2q}{q-1}. 
\end{eqnarray*} 
Note that applying a suitable  mapping class yields $ \smallD{2}{1}{2q}{q+1} =  \smallD{2}{1}{2q}{q-1}  = \begin{cases} 1,& q \mbox{ odd} \\ 0, &q \mbox{ even} \end{cases}$. 

\end{proof}

 We can continue applying the Recursion Rule indefinitely to obtain various formulas for families of $\smallD{p}{q}{r}{s}$, and this is the content of the algorithm in the next section.   While these formulas are computable, they do get more complicated very rapidly.   There did not seem any utility in continuing in this vein unless we had a way to package the terms nicely, like the $T$-Chebyshevs in the Product-to-Sum Formula in the closed torus case.   The authors unfortunately were unable to find such a suitable basis in the one-hole torus case.

 \begin{remark}
 Let $S_k$ denote the $k$th Chebyshev polynomial of the second kind.  It is defined recursively by $S_0(x) = 1$, $S_1(x) = x$, and $S_{k}(x) = x S_{k-1}(x) - S_{k-2}(x)$ for $k \geq 2$.   Analogously, one may define $\pqvect{p}{q}_S = S_{\gcd(p, q)} \pqvect{\frac{p}{\gcd(p, q)}}{\frac{q}{\gcd(p, q)}} $.  
 We experimented with writing the discrepancy using the $S$-Chebyshevs, for example 
\[  \bigD{p}{1}{0}{1}
= \sum_{k=0}^{\lfloor \frac{p}{2} \rfloor}  A^{-(p-2k)}  \pqvect{p-2k}{0}_S\] 
However, at least to the authors, we did not find any significant advantage of doing so. 
\end{remark}

\section{Fast Algorithm for computing the  Discrepancy} \label{algorithm}

We summarize the algorithm of Theorem \ref{algthm}, with subroutines appearing below.  Note that by convention, $p \geq 0$.  

\begin{algorithm}
 \caption{Computes $D(p,q; r, s)$}
 \label{mainalgorithm}
\begin{algorithmic}
\Require $p \geq 0$
    \State (Preprocessing)
\end{algorithmic}

\begin{enumerate}
\item If $|ps- qr| \leq 1$, then Return $0$.  \\
\item  If $|ps- qr| \geq 2$, Apply mapping class $\phi$ from Lemma~\ref{Dehn}. Return $\phi_*\smallD{p}{q}{r}{s} = \smallD{p_1}{q_1}{0}{s_1}$ where $0\leq q_1<p_1$. 
\begin{itemize}
\item If $q_1 = 0$, Return $0$.\\
\item If $q_1 \geq 1$, Set $q_2 = \min(q_1, p_1-q_1)$:
\begin{enumerate}
\item If $q_2 = q_1$, Return  $\smallD{p_1}{q_2}{0}{s_1}$.\\

\item If $q_2 = p_1 - q_1$, Return $\opp \left( \smallD{p_1}{q_2}{0}{s_1} \right) $ from Lemma~\ref{opp}.
\end{enumerate}
\end{itemize}
\end{enumerate}
\hrulefill

\begin{algorithmic}
    \State (Computing $\smallD{p_1}{q_2}{0}{s_1}$)
\end{algorithmic}
\begin{enumerate}
\item If  $s_1 =1$,  do Algorithm \ref{subroutine:s=1} to compute $\smallD{p_1}{q_2}{0}{1}$. \\

\item If $s_1 \geq 2$, 
\begin{itemize}
\item do Algorithm \ref{subroutine:s=1} to compute $\smallD{p_1}{i}{0}{1}$ for $i = 1, \ldots  \lfloor \frac{p_1}{2} \rfloor$.
\\
\item 
do Algorithm \ref{subroutine:bigs}  to compute $\smallD{p_1}{q_2}{0}{k}$ for $k = 2, \ldots s_1$.\\
\end{itemize}
\end{enumerate}
\end{algorithm}


\begin{algorithm}\caption{Subroutine to compute $D(p,q; 0, 1)$} \label{subroutine:s=1}
\begin{algorithmic}
 \Require $1 \leq q \leq p $   
\end{algorithmic}
\begin{enumerate}
\item Initialize with 
	\begin{itemize}
	\item $ \smallD{1}{q}{0}{1}   =  0$ for all $q$.  
	\item $\smallD{2}{q}{0}{1} = \begin{cases} 0& q = 0 \pmod{2} \\ 1& q= 1 \pmod{2} \end{cases}$. 
	\item $ \smallD{p}{0}{0}{1}  = 0$ for all $p$.  
	\end{itemize}
\item For  $n = 1, \ldots q$: 
\item Initialize with  $\smallD{n}{n}{0}{1} = 0$ and $\smallD{n+1}{n}{0}{1} = \opp \left( \smallD{n+1}{1}{0}{1} \right)$
\item For $m = n+2, \ldots p$: 
\begin{itemize}
\item compute $\smallD{1}{0}{m}{n-1}$, $\smallD{1}{0}{m}{n}$,  and $\smallD{1}{0}{m}{n+1}$ using Lemma \ref{extraterms}
\item compute $\smallD{m}{n }{0}{1}$ using Corollary \ref{1-recursion} and Equation (\ref{def:D}).

\end{itemize}
\end{enumerate}
\end{algorithm}


\begin{algorithm}\caption{Subroutine to compute $\bigD{p}{q}{0}{s}$} \label{subroutine:bigs}
\begin{algorithmic}
 \Require $s \geq 2$   
\end{algorithmic}

\begin{enumerate}
\item Initialize with $\smallD{p}{i}{0}{1}$ for $i = 1, \ldots  \lfloor \frac{p}{2} \rfloor$ (from Algorithm \ref{subroutine:s=1})
\item For m = 2, \ldots s: 
\begin{itemize}
\item compute $\smallD{p}{q+m}{0}{1}$ and $\smallD{p}{q-m}{0}{1}$ using Lemma \ref{tw}
\item compute $\smallD{p}{q}{0}{m}$ using Corollary \ref{s-recursion}. 
\end{itemize} 
\end{enumerate} 
\end{algorithm}


\newpage

We introduce some notation for ease of exposition.  For $a, b, p, q \geq 0$,  we write  $\smallpqvect{a}{b} \leq  \smallpqvect{p}{q}$ whenever $ a \leq p$ and $ b  \leq q$.   Similarly, a linear combination of torus links is smaller than $\smallpqvect{p}{q}$ if every summand of the linear combination is smaller than $\smallpqvect{p}{q}$.

\begin{proposition} \label{Dpq01}
The Recursion Rule computes $\bigD{p}{q}{0}{1}$ for any integers $p, q \geq 0$.

\end{proposition}

\begin{proof}
From Lemmas \ref{smallcrossings} and \ref{q=0}, we have that $\smallD{p}{q}{0}{1}$ is computable for  $p=0, 1$ or $q=0$.  Hence assume from now on that $p \geq 2$ and $q \geq 1$.   We will use induction to prove a slightly stronger statement, that $\smallD{p}{q}{0}{1}$ is computable and that no term in it is larger than  $\smallpqvect{p-2}{q-1}$.   Lemma \ref{smallcrossings2} provides the base case $p = 2$, and  Lemma \ref{q=1} provides $q= 1$.

Now fix $p \geq 3$ and $q \geq 2$.  Assume that for $n= 0, \ldots q$, $\smallD{p}{n}{0}{1}$ has been computed and that $\smallD{p}{n}{0}{1} \leq \smallpqvect{p-2}{n-1}$ for all $p$.  Also assume that for $m = 0, \ldots p$, $\smallD{m}{q+1}{0}{1}$ has been computed and that $\smallD{m}{q+1}{0}{1} \leq \smallpqvect{m-2}{q}$.  We wish to show that the five terms on the right of the Recursion Rule from Corollary \ref{1-recursion}
\begin{eqnarray*}
A^{q+1} \bigD{p+1}{q+1}{0}{1} 
&=& \pqvect{1}{0} \bigD{p}{q+1}{0}{1} - A^{-(q+1)} \bigD{p-1}{q+1}{0}{1} \\
& & +   \quad A^{-p} \bigD{1}{0}{p}{q} - \bigD{1}{0}{p}{q+1} \pqvect{0}{1} + A^p \bigD{1}{0}{p}{q+2}
\end{eqnarray*}
can be obtained from previous terms and that each term is no larger than $\smallpqvect{p-1}{q}$.

For the second term  $\smallD{p-1}{q+1}{0}{1}$, the induction hypothesis applies directly.

For the third term,  observe that  $\smallD{1}{0}{p}{q}$ is the reverse of $\smallD{q}{p}{0}{1}$.   Thus, one is computable if and only if the other is computable, and the inequality $\smallD{1}{0}{p}{q} \leq \smallpqvect{p-1}{q}$ is equivalent to $\smallD{q}{p}{0}{1} \leq \smallpqvect{q}{p-1}$.    If $p \leq q $, then the induction hypothesis applies, and in particular, $\smallD{q}{p}{0}{1} \leq \smallpqvect{q-2}{p-1} \leq \smallpqvect{q}{p-1}$.    If $p \geq q$, then we first reduce $p$ mod $q$.  We write $p = aq + p'$ where $a \in \Z$ and $0 \leq p' < q$, so that the induction hypothesis applies and $\smallD{q}{p'}{0}{1} \leq \smallpqvect{q-1}{p'-1}$.  We can then compute $\smallD{q}{p}{0}{1}$ by applying a Dehn twist $a$ times as in Lemma \ref{Dehn}, and $\smallD{q}{p}{0}{1} \leq \smallpqvect{q-1}{a(q-1) + p'-1} \leq \smallpqvect{q}{p-1}.$

The fifth term is similar.  $\smallD{1}{0}{p}{q+2}$ is the reverse of $\smallD{q+2}{p}{0}{1}$, and it is equivalent to show $\smallD{q+2}{p}{0}{1} \leq \smallpqvect{q}{p}$. If $p \leq q$, then this is covered by the induction hypothesis.   If $p = q+1$, apply a transformation and then use the opposition transformation from Proposition \ref{opp} (alternatively, we can appeal to Equation (\ref{q=p-1})).  For $p \geq q+2$,  first reduce $p$ mod $q+2$ and then apply the induction hypothesis.

For the first term, the induction hypothesis implies that  $ \smallD{p}{q+1}{0}{1}  =  \displaystyle \sum_{\substack{0 \leq m \leq p-2 \\ 0 \leq n \leq q}} c_{mn} \Tvect{m}{n}$, with the known coefficients $c_{mn} \in \Z[A, A{-1}, \eta]$.  By Equation (\ref{def:D}),
\[ \smallpqvect{1}{0} \smallD{p}{q+1}{0}{1} 
=  \sum_{\substack{0 \leq m \leq p-2 \\ 0 \leq n \leq q}} c_{mn}( A^n \smallTvect{m+1}{n} + A^{-n} \smallTvect{m-1}{n} + \eta \smallD{1}{0}{m}{n} )\]
Since $m \leq p-2$ and $n \leq q$, we have that $\smallTvect{m+1}{n},  \smallTvect{m-1}{n} \leq \smallpqvect{m-1}{n}$.    Also, because $0 \leq b \leq q$, the induction hypothesis implies that, possibly after reduction of $m$ mod $n$, the discrepancy  $\smallD{n}{m}{0}{1}$ is computed and $\smallD{n}{m}{0}{1} \leq \smallpqvect{q}{p-1}$.  By taking the reverse, it follows that $\smallD{1}{0}{m}{n}$ is computable and  $\smallD{1}{0}{m}{n} \leq \smallpqvect{p-1}{q}$.

For the fourth term, the argument is similar. 
The induction hypothesis gives $\smallD{q+1}{p}{0}{1} \leq \smallpqvect{q-1}{p-1}$, possibly after reduction of $p$ mod $q+1$.  Thus  $\smallD{1}{0}{p}{q+1} = \displaystyle \sum_{\substack{0\leq m \leq p-1 \\0 \leq n \leq q-1}} c_{mn} \smallTvect{m}{n} $, and 
\[ \smallD{1}{0}{p}{q+1} \smallpqvect{0}{1}  
 =  \sum_{\substack{0 \leq m \leq p-1 \\ 0 \leq n \leq q}} c_{mn} (A^m \smallTvect{m}{n+1} + A^{-m} \smallTvect{m}{n-1} + \eta \smallD{m}{n}{0}{1} ).\] 
Because $0 \leq b \leq q$, each term on the right is known, and no term is larger than $\smallpqvect{p-1}{q}$.

\end{proof}

\begin{remark} In the course of the proof above, we showed that when $p \geq 2, q \geq 1$,  we have the discrepancy  $\smallD{p}{q}{0}{1} = \sum_{a, b} c_{ab} \smallTvect{a}{b}$ where the sum is taken over $0\leq a \leq p-2$ and $ 0 \leq b \leq q-2$.  See Theorems 4.6 and  4.7 of \cite{bakshi4-holedsphere} for comparable statements for the 4-holed sphere, which were proved by a careful analysis of the intersection numbers and without the organizing principle provided by the Recursion Relation. 
\end{remark}


\begin{remark}
Proposition \ref{Dpq01} proves the validity of  Algorithm \ref{subroutine:s=1} for computing $\smallD{p}{q}{0}{1}$ for any $p, q \geq 0$.  However, as can be seen by a careful analysis of the proof,  we can get substantial computational savings at several key steps.  First, one can apply Dehn twists, as in Corollary \ref{extraterms}, to reduce $q$ mod $p$.  One can also apply Proposition \ref{opp} to ensure that $q \leq \lfloor p/2 \rfloor$.  This is incorporated into Steps 3 and 5 of Algorithm \ref{mainalgorithm}.

 In addition, for fixed $q$, prior to an iteration of the Recursion Rule to compute $\smallD{p}{q+1}{0}{1}$, Algorithm \ref{subroutine:s=1} first computes and stores all values of $\smallD{1}{0}{m}{q}$ from values of its reverse $\smallD{q}{m}{0}{1}$ from earlier iterations by using Dehn twists to reduce $m$ mod $q$.  The algorithm also pre-computes $\smallD{1}{0}{m}{q+1 }$ and $\smallD{1}{0}{m}{q+2}$ in the same way.  Note that we may have to apply Proposition \ref{opp} to reduce $\smallD{q+1}{q}{0}{1}$, $\smallD{q+2}{q+1}{0}{1}$ and $\smallD{q+2}{q}{0}{1}$ to reach stored values.  We will see that this will give computational savings in  Proposition \ref{complexity}. 
\end{remark}

\begin{proposition}
The Recursion Rule computes  $\bigD{p}{q}{r}{s}$ for all $p, q, r, s$. 
\end{proposition}

\begin{proof}
By applying a mapping class, we reduce to the the case of $\smallD{p}{q}{0}{s}$ where $0 \leq q \leq \lfloor p/2 \rfloor$.  Proposition \ref{Dpq01} proves the validity of the subalgorithm when $s=1$.  For larger $s$, let us assume that we have already computed  and stored $\smallD{a}{i}{0}{1}$ for $a = 1, \ldots p$ and $i = 0, \ldots, \lfloor p/2 \rfloor$. Moreover, from the proof of Proposition \ref{Dpq01} we also have that $\smallD{a}{i}{0}{1} \leq \smallpqvect{a-2}{i-1}$.

Fix $s$.  The base case is given by Proposition \ref{Dpq01}.  Assume that we have $\smallD{p}{q}{0}{k}$ for all $k = 1, \ldots s$ and that no term of $\smallD{p}{q}{0}{k}$ is larger than $\smallpqvect{p-2}{q + k-1}$. 
 By the Recursion Relation from Corollary \ref{s-recursion}, we have
\[
\smallD{p}{q}{0}{s+1} = \smallD{p}{q}{0}{s} \cdot \smallTvect{0}{1} - \smallD{p}{q}{0}{s-1} + A^{ps} \smallD{p}{q+s}{0}{1} +A^{-ps} \smallD{p}{q-s}{0}{1}  
\] 
The last two terms can be obtained from the given values of $\smallD{p}{q}{0}{1}$, by applying Dehn twists as in Lemma \ref{Dehn} to reduce $q+s$ (resp. $q-s$) mod $p$.

The second term follows directly from the induction hypothesis.

For the first term,  the induction hypothesis implies $\smallD{p}{q}{0}{s} = \displaystyle \sum_{\substack{0\leq m \leq p-2 \\0 \leq n \leq q + s-1}} c_{mn} \smallTvect{m}{n} $. Thus
\[ \smallD{p}{q}{0}{s} \smallpqvect{0}{1}  
 =  \sum_{\substack{0 \leq m \leq p-2 \\ 0 \leq n \leq q+s-1}} c_{mn} (A^m \smallTvect{m}{n+1} + A^{-m} \smallTvect{m}{n-1} + \eta \smallD{m}{n}{0}{1} ).\] 
 Each of the terms on the right may be computed by simple transformations from known values, and none is larger than $\smallpqvect{p-2}{q+s}$, as desired. 
\end{proof}


\begin{proposition}  \label{complexity} The time and space complexity of the algorithm computing $D\begin{pmatrix}
p & r \\ q & s \end{pmatrix}$ is $O\bigg(\begin{vmatrix} p & r\\q&s\end{vmatrix}^6\bigg).$   
\end{proposition}

\begin{proof}
Using coefficients from the mapping class $\phi$ in Lemma 3.2, it is sufficient to consider $D\begin{pmatrix}
p & 0 \\ q & s \end{pmatrix}$. With a complexity of $\widetilde{O}(\log (\min(r,s)))$, the Euclidean algorithm allows us to assume that $s, r$ have similar orders of magnitude.

We will estimate the space and time complexity for computing $D\begin{pmatrix}
p & 0 \\ q & 1 \end{pmatrix}$, then apply Corollary 3.4 to obtain $D\begin{pmatrix}
p & 0 \\ q & s \end{pmatrix}$.

First, suppose that the value of all $D\begin{pmatrix}
u & 0 \\ v & 1 \end{pmatrix}$ for which $u \leq p$ and $v \leq q$ are already computed. In the five-term recursion of Corollary 3.3, the second, third, and fifth term involve accessing those precomputed values, all of which take space $O(pq)$ to store. Looping over the values once therefore takes $O(pq)$ time for summation and simplifications. \\

The first term of the five-term recursion formula involves computing the expression 
$$\begin{pmatrix}
1 \\ 0 \end{pmatrix}_T \cdot D\begin{pmatrix}
p & 0 \\ q-1 & 1 \end{pmatrix}$$
which involves multiplication with stored values $$D\begin{pmatrix}
p & 0 \\ q-1 & 1 \end{pmatrix} = \sum_{0\leq m \leq p-2, 0\leq n \leq q-1} c_{mn} \begin{pmatrix}
m \\ n \end{pmatrix}_T $$
Overall, this gives us the expression 
$$ \begin{pmatrix}
1 \\ 0 \end{pmatrix} \cdot D\begin{pmatrix}
p & 0 \\ q-1 & 1 \end{pmatrix} = \sum_{0\leq m \leq p-2, 0\leq n \leq q-1} c_{mn}(A^n \begin{pmatrix}
m+1 \\ n \end{pmatrix}_T + A^{n-1} \begin{pmatrix}
m-1 \\ n \end{pmatrix}_T + \eta D\begin{pmatrix}
1 & m \\ 0 & n \end{pmatrix})$$
This involves at most $pq$ of $D\begin{pmatrix}
1 & m \\ 0 & n \end{pmatrix}$'s, each with at most $pq$ terms. Therefore it takes $O(p^2q^2)$ operations to substitute and simplify the terms. The argument is similar for the fourth term in the recursion formula.\\

To obtain $D\begin{pmatrix}
p & 0 \\ q & s \end{pmatrix}$ from $D\begin{pmatrix}
p & 0 \\ q & 1 \end{pmatrix}$, one needs to apply the formula from Corollary 3.4 for each $D\begin{pmatrix}
p & 0 \\ q & i \end{pmatrix}$ from $i = 2, ..., s$. Notice that the first multiplicative term takes $O(p^2q^2)$ to compute, which dominates the running time of each iteration. Counting the additional space used to store these terms, this gives an overall complexity of $O(p^2q^2s)$.\\

Finally, to satisfy the assumption that all $D\begin{pmatrix}
u & 0 \\ v & 1 \end{pmatrix}$ for $u \leq p$ and $v \leq q$ are computed, we compute all such $D\begin{pmatrix}
u & 0 \\ v & 1 \end{pmatrix}$'s, with the explicit order given by Algorithm 2. This gives another loop of $O(pq)$ terms, which brings both the time and space complexity to store all items to $O(p^3q^3s)$. Since we have $q = \lfloor p/2 \rfloor$ in the worst case, we obtain an upper bound of $O\bigg(\begin{vmatrix} p & r\\q&s\end{vmatrix}^6\bigg)$ for the overall complexity. 
\end{proof}

{\bf Acknowledgements.} 
The authors would also like to thank the Claremont McKenna College mathematics department for their support in this research, as well as Wade Bloomquist, Francis Bonahon, Han-Bom Moon and Charles Frohman for helpful conversations.   They also thank the anonymous referee for a careful reading of the paper.  
Preliminary research used Mathematica code for computing multiplication written by Jonathan Hahn and Collin Hazlett while undergraduates at Carleton College.   The authors were supported in part by NSF Grants DMS-1906323 and DMS-2305414.


\renewcommand{\refname}{References Cited}

\bibliographystyle{plain}

 \bibliography{PuncturedTorusRefs}

\end{document}